\documentclass[12pt]{amsart}
\usepackage{amssymb,amsmath,amsfonts}
\usepackage[dvips]{epsfig}
\usepackage{times}
\usepackage{graphicx}
\usepackage{color}

\makeatletter\@addtoreset {equation}{section}\makeatother

\newtheorem{theorem}{Theorem}[section]
\newtheorem{lem}[theorem]{Lemma}
\newtheorem{prop}[theorem]{Proposition}

\theoremstyle{definition}
\newtheorem{example}[theorem]{Example}
\newtheorem{question}[theorem]{Question}

\def\Chi{\mathcal X}
\def\E{\mathcal E}
\def\eps{\varepsilon}

\def\eps{\varepsilon}

\def\RR{{\mathbb R}}
\def\V{\mathcal V}
\def\W{\mathcal W}

\begin{document}

\setcounter{page}{1}

\title{A stability result for Riesz potentials in higher dimensions}

\author{Almut Burchard and Gregory R. Chambers}

\date{July 22, 2020}

\begin{abstract} We prove a stability estimate,
with the optimal quadratic error term, for 
the Coulomb energy of a set in $\RR^n$ with $n \geq 3$. 
This estimate extends to a range of Riesz potentials.

\end{abstract}
\maketitle

\section{Introduction}
\label{sec:intro}

Let $\phi$ be a strictly
radially decreasing, nonnegative measurable function
on $\RR^n$ that vanishes at infinity,
and consider the convolution functional
$$
\E(f) = \int_{\RR^n}\int_{\RR^n} f(x) f(y) \phi(x-y)
\, dxdy\,,
$$
defined for measurable functions $f$ on $\RR^n$
that satisfy a suitable integrability condition.
It is well-known that the value of $\E$
can only increase under symmetric decreasing rearrangement
of $f$~\cite{R-1930,S-1938}: If $f^*$ is the 
function equimeasurable with $|f|$
that decreases radially about the origin, then
the Riesz--Sobolev inequality implies that
$$
\E(f^*) \ge \E(f)\,.
$$
Equality occurs (for a finite, non-zero value of $\E$
and nonnegative $f$) if only if $f$ itself is symmetric decreasing
about some point $x_0\in\RR^n$, that is, 
$f(x)=f^*(x-x_0)$~\cite{L-1976}.
A natural question is whether near-equality
implies that $f$ must be close to a translate of $f^*$? If so, how close
must it be?

In this paper, we investigate near-equality cases
in the Riesz--Sobolev inequality when $\phi=\phi_\lambda$ is a
{\em Riesz potential}
\begin{align}
\label{eq:def-Riesz}
\phi_\lambda (x)= \frac{1}{c_\lambda}|x|^{-(n-\lambda)}\,,
\end{align}
and $f=\Chi_A$ is the characteristic function
of a set $A\subset\RR^n$ of finite positive volume $|A|$.
Here, $0<\lambda<n$, and $c_\lambda$ is a particular
normalizing constant, see Eq.~\eqref{eq:def-c} below.
The symmetric decreasing rearrangement of $\Chi_{A}$ 
is the characteristic function of $A^*$,
the centered ball of the same volume as $A$. 
With a slight abuse of notation, we write
\begin{equation}
\label{eq:def-E}
\E_\lambda(A)= \int_A \int_A \phi_\lambda(x-y)\, dxdy\,.
\end{equation}
Our main result provides a lower bound on the {\em deficit}
\begin{equation}
\label{eq:def-delta}
\delta(A):= \left(\frac{|B^n|}{|A|}\right)^{2-\frac{\lambda}{n}}
\bigl (\E(A^*)-\E(A)\bigr)
\end{equation}
in terms of the {\em Fraenkel asymmetry}
\begin{equation}
\label{eq:def-alpha}
\alpha(A):= \frac{|B^n|}{|A|} 
\inf_{x\in\RR^n} 
\bigl\{ | A\Delta (x+A^*)|\bigr\}\,.
\end{equation}
Here, $x+A^*$ is the ball of the same volume as $A$, centered
at $x$, and $\Delta$ denotes the symmetric difference of
two sets. Both the deficit and the asymmetry are invariant under 
translation, rotation, and dilation.

\begin{theorem} [Stability] 
\label{thm:main}
Fix the dimension $n\ge 2$.
For $\lambda\in (1,n)$, let $\phi_\lambda$ be the Riesz potential
defined by Eqs.~\eqref{eq:def-Riesz} and~\eqref{eq:def-c},
and let $\E_\lambda$ be as in Eq.~\eqref{eq:def-E}.
There exists a constant $C_\lambda>0$ such that
$$
\delta(A)\ge C_\lambda \alpha^2(A)
$$
for every subset $A\subset\RR^n$ of finite positive volume.
\end{theorem}

The exponent $2$ is best-possible, see Example~\ref{ex:ring},
but we do not determine
the value of the sharp constant $C_\lambda$.
Theorem~\ref{thm:main} extends our previous results 
on the Newton potential from~\cite{BC-2015}.
In that work, we used techniques that had been
developed for the quantitative isoperimetric 
inequality~\cite{FMP} along
with reflection positivity of the Coulomb kernel
to prove Theorem~\ref{thm:main} for $\lambda=2$ in dimension $n=3$. 
In higher dimensions, we obtained a weaker
inequality (with a non-sharp exponent) from
Talenti's comparison principle for solutions of
Poisson's equation~\cite{T}. The proof
of Theorem~\ref{thm:main} that we present here 
instead uses an approach of Fuglede to
obtain explicit estimates for near-spherical sets, 
in combination with lengthy but straightforward
global rearrangements.

Shape optimization problems involving convolution
functionals appear in physical and biological models for 
pair interactions between large numbers of particles
or individuals. Geometric stability results for
such non-local functionals have many potential applications,
from dynamical stability for isotropic steady states
in stellar dynamics~\cite{BG-2004}, 
to the construction of continuum limits in statistical 
mechanics~\cite{CCELM-2009}, and flocking in biological aggregation 
models~\cite{FL20192}. Regardless of their importance,
such problems are not as well-understood as the
classical isoperimetric inequality and other 
inequalities for gradient functionals. 
Fewer explicit estimates are available for
the integral equations which characterize optimal shapes
than for elliptic PDE arising from gradient functionals.
Very recently, a number of results have begun to
address these questions.

The fundamental stability question 
for the Riesz--Sobolev inequality, in the case
where all three functions in the convolution integral
are symmetrized simultaneously, has been settled
by M. Christ in a series of papers since 2013.
In~\cite{C-2017}, he proves a sharp result, where the
geometric asymmetry of a triple of sets is controlled by the 
square root of their deficit in the Riesz--Sobolev inequality. 
This estimate is rather delicate because it can hold
only when the three sets are comparable in size.
Frank and Lieb in~\cite{FL2019} 
extend Christ's results (in the case
of a radially decreasing integral kernel) 
from 
sets to densities taking values in the interval $[0,1]$.
In a different direction, Figalli and Jerison 
obtain geometric stability results
for the Brunn--Minkowski inequality~\cite{FJ-2017},
an affine invariant inequality for the volume
of sum sets, which can be seen as a limiting
case of the Riesz--Sobolev inequality.
For the important case when one summand is a ball, the Brunn--Minkowski
inequality becomes a non-local isoperimetric inequality.
Here, sharp stability estimates are due to Figalli, Maggi, 
and Mooney~\cite{FMM-2016}. 
While the results described above
are motivated by insight from
additive combinatorics, convex geometry, and
geometric measure theory, their
proofs tend to rely on direct estimates on how different
parts of a set contribute to the integral functional
under consideration.
The approach of Fuglede that we employ here has 
also been used to give a new proof of the quantitative 
isoperimetric inequality; see the article~\cite{F} 
by N. Fusco for a survey of all of these techniques.

We close out this section with some open questions.

\begin{question}
	Does Theorem~\ref{thm:main} extend to $\lambda\le 1$?  
\end{question}
The spherical integral that we use as a toy model
for $\E_\lambda$ makes sense only for $\lambda>1$ 
(see Section~\ref{sec:outline}, particularly
Eq.~\ref{eq:def-M}))
We suspect that the case of $\lambda = 1$ can be resolved by
taking appropriate limits in our proof.  However, 
for $\lambda \in (0,1)$, our method breaks down,
and a new idea is needed.

\begin{question}
What is the analogous result for functions?  
\end{question}
As discussed in the opening lines,
the Riesz--Sobolev inequality implies
that  the functional $\E_\lambda(f)$
can only increase (for $f\ge 0$) if $f$ is replaced by its
symmetric decreasing rearrangement, $f^*$. It increases strictly,
unless $f$ is a translate of $f^*$.
Measuring the deficit of $f$ as above by
$\delta(f)=\E_\lambda(f^*)-\E_\lambda(f)$,
what is the correct measure of asymmetry, and
what is the correct sharp stability inequality? 

It has been conjectured that
$$
\E_\lambda(f^*)-\E_\lambda(f) \ge C_\lambda
\inf_{a\in \RR^n}
\iint (f^*(x)-f(x\!-\!a))(f^*(y)-f(y\!-\!a))\phi_\lambda(x-y)\, dxdy
$$
for some positive constant $C_\lambda$ (which 
depends on $n$)~\cite{BG-2004}. The right hand
side in this inequality equals the squared distance of
$f$ from the translates of $f^*$ in the negative Sobolev
space $H^{-\frac\lambda 2}$, a space of distributions that
contains $L^{\frac{2n}{n+\lambda}}$.
It is known that this distance is small
whenever the left hand side is small, but no explicit 
bounds are known.

\medskip 
The energy functional $\E_\lambda$ is closely connected with
the notion of \emph{Riesz capacity} of a compact
subset $A \subset \mathbb{R}^n$, given by
$$ {\rm Cap}_{n-\lambda}(A) = \sup_\mu 
\left(\int \int \phi_\lambda (x-y)\,d\mu(x) d\mu(y)\right)^{-1}\,,
$$
where the supremum is taken over all probability
measures $\mu$ supported on $A$.
This quantity has been studied extensively in the literature;
it is  used in Harmonic Analysis and Metric Geometry 
to obtain lower bounds on the Hausdorff dimension of sets.
For further discussion, we direct the reader 
to Chapter 6 of~\cite{F-1986}.

\begin{question} Among sets of given volume,
are near-minimizers of the Riesz capacity 
close to balls?
\end{question}

For the special case of the Newton potential ($\lambda=2$) in
dimension $n\ge 3$, the Riesz capacity agrees with
the more common notion of capacity,
defined by minimizing the Dirichlet integral
of the potential generated by the mass distribution
under suitable constraints. For that notion of capacity,
balls are indeed minimal, and stability follows
via the co-area formula from
from results on the classical isoperimetric
inequality~\cite{mazya-2003,fmp-2009}.

Balls are known to be the unique minimizers of the
Riesz capacity also in the range 
$\lambda\in (0,2)$~\cite{W-1983,betsakos-2004,mendez-2006};
it is not clear whether this continues to hold for $\lambda>2$. 
Theorem~\ref{thm:main} does not apply, because the equilibrium 
measure is generally far from uniform, concentrating
on the boundary for $\lambda\ge 2$ and near the boundary
for $\lambda\in (0,2)$. Still, the estimates
in Sections~\ref{sec:V} and~\ref{sec:W} may prove
useful.

\subsection*{Acknowledgments} 
We would like to thank Alessio Figalli 
for suggesting to look at the approach of Fuglede.
The first author was supported in part by
NSERC Discovery Grant 311685, and the second author was
supported in part by NSF Grant DMS-190654.

While preparing this manuscript, the authors became aware
that Theorem~4 in preprint just posted by Frank and Lieb 
~\cite{FL20192} implies the main stability result
of this article.  Since their proof proceeds 
along different lines, and this
article explores other aspects of the problem, we
feel that the results presented here 
are still relevant and of interest.

\section{Outline of the proof}
\label{sec:outline}


We work in Euclidean space $\RR^n$ of a fixed dimension
$n\ge 2$ (which we routinely suppress in the notation).
The volume and surface area of the unit ball $B^n$
are given by
$$|B^n| = \frac{\pi^{\frac{n}{2}}}{\Gamma(\tfrac{n}{2}+1)}\,,
\qquad 
|S^{n-1}| = n|B^n| = \frac{2 \pi^{\frac{n}{2}}}{\Gamma(\tfrac{n}{2})}\,.
$$
We use lowercase
Roman symbols to denote points in Euclidean space (e.g., $x\in\RR^n$),
and Greek symbols for points on the unit sphere (e.g., $\xi\in S^{n-1}$).
On the unit sphere, we denote by
$|\xi-\eta|$ the chordal distance, by $d\xi$ 
integration against the 
uniform measure inherited from the Lebesgue measure on $\RR^n$,
and by $\|\cdot\|$ the $L^2$-norm.

For $0<\lambda<n$,
the {\em Riesz potentials} are defined by
$$
\phi_\lambda (x)= \frac{1}{c_\lambda}|x|^{-(n-\lambda)}
$$
on $\RR^n$ (see \cite{Stein-book} and \cite{Landkof-book}). 
The constant is conventionally chosen as
\begin{equation}
\label{eq:def-c}
c_\lambda= 2^{\lambda} \pi^{\frac{n}{2}}
\frac{\Gamma(\tfrac{\lambda}{2})}{\Gamma( \tfrac{n-\lambda}{2})}\,.
\end{equation}
With this normalization, their Fourier transform is given by
$\hat \phi_\lambda(p) = (2\pi |p|)^{-\lambda}$ 
(see, for example,~\cite[Theorem 5.9]{LL-book}),
and $\phi_\lambda * \phi_\mu = \phi_{\lambda + \mu}$.
Since $\phi_\lambda$ has a positive Fourier transform,
the convolution functional is {\em positive definite}, that is,
$$
\iint f(x) f(y)\phi_\lambda(x-y)\, dxdy>0
$$ 
for any non-zero function $f$ such that the integral exists.
The value of this double integral equals
$\|f\|^2_{H^{-\lambda}}$, the squared 
Sobolev norm of fractional order $-\lambda$.
The case $\lambda=2$
$$
\phi_2(x) = \frac{1}{(n-2)|S^{n-1}|} |x|^{-(n-2)}\,,
\qquad n\ge 3\
$$
is the {\em Newton potential},
which plays a special role in Mathematical Physics, with
many applications to gravitation and electrostatics.
It is also the fundamental solution of the (negative)
Laplacian.  

Consider minimizing the deficit among all
sets $A\subset\RR^n$ of given asymmetry~$\alpha>0$.
Since both $\delta$ and $\alpha$ are bounded functionals, 
it suffices to analyze the situation
when $\alpha$ is small. One
may also scale $A$ to a given volume,
and translate it so that the infimum in the definition
of the Fraenkel asymmetry occurs at $x=0$.

Suppose that $A$ is squeezed between two balls,
\begin{equation}
\label{eq:ring}
e^{-\eps} B^n\subset A\subset e^\eps B^n\,,
\end{equation}
where $\eps$ is a small positive number that will
later be chosen as a function of the asymmetry (with
$\eps\to 0$ as $\alpha \rightarrow 0$, see Proposition~\ref{prop:reduce}).
Denote by $\Phi_\lambda := \Chi_{B^n}*\phi_\lambda$
the potential of the unit ball.  We expand the difference as
$\E(B^n)-\E(A)=\V(A)-\W(A)$, with 
\begin{align}
\label{eq:def-V}
\V(A) = 2\int \bigl(\Chi_{B^n}(x)-\Chi_A(x)\bigr) 
\Phi_\lambda(x)\, dx\,,
\end{align}
and 
\begin{align}
\label{eq:def-W}
\W(A)= \iint\bigl(\Chi_{B^n}(x)-\Chi_A(x)\bigr)
\bigl(\Chi_{B^n}(y)-\Chi_{A}(y)\bigr)
\phi_\lambda(x-y)\, dydx\,.
\end{align}
Since $\Phi_\lambda$ is strictly radially decreasing, 
$\V(A)>0$ by the bathtub principle.
On the other hand, $\W(A)> 0$ since the
Riesz potential is positive definite.
We seek a positive lower bound on the difference
$\V-\W$ in terms of the asymmetry.

Viewing $\V$ as the first variation of $\E$ about $B^n$, and 
$\W$ as the second variation,
one would expect $\V(A)$ to vanish linearly and $\W(A)$ to vanish 
quadratically as $\alpha\to 0$,
and thus $\E(B^n)-\E(A)$ to be comparable to $\V(A)$ 
for $A$ sufficiently close to $B^n$.
However, if $A$ has the same volume as $B^n$,
then in fact $\V(A)$ also vanishes quadratically, 
since the value of $\Phi_\lambda$ is almost constant
in a neighborhood of the unit sphere.  In order to obtain 
a positive lower bound on $\V(A)-\W(A)$,
the competing terms have to be estimated with care.

Our strategy, based on work of Fuglede~\cite{F-1989},
is to approximate $\V$ and $\W$ by
integral functionals  on the unit sphere, 
and then solve the resulting minimization problem
by expanding in spherical harmonics.
For each direction $\xi\in S^{n-1}$, consider the contribution
of the ray through $\xi$ to the
volume of $A\setminus B^n$ and $B^n\setminus A$, respectively, 
given by
\begin{equation}
\label{eq:def-M}
M_+(\xi)= \int_1^\infty \Chi_{A}(r\xi)\, r^{n-1}\, dr\,, \quad 
M_-(\xi)= \int_0^1 \
(1-\Chi_{A}(r\xi))\, r^{n-1}\, dr\,.
\end{equation}
By construction, $|A|-|B^n|=\int_{S^{n-1}}(M_+-M_-)\,d\xi$.
We scale $A$ to the same volume as $B^n$ and center
it near the origin to obtain
\begin{equation}
\label{eq:constraints}
\begin{split}
& \int_{S^{n-1}} M_+(\xi)\, d\xi = \int_{S^{n-1}} M_-(\xi)\, d\xi
\ge \frac{\alpha}{2}\,,\\
& \int_{S^{n-1}} \xi \bigl(M_+(\xi)-M_-(\xi)\bigr)\, d\xi = 0\,.
\end{split}
\end{equation}

We next express $\delta(A)=\V(A)-\W(A)$ in terms of $M_+$ and $M_-$.
The first variation decreases if mass from $A$ is moved 
inwards along the rays towards the origin. As $\eps\to 0$, we 
obtain
\begin{eqnarray}
\label{eq:V}
\V(A) 
& \gtrapprox \ -\Phi_\lambda\Big\vert_{S^{n-1}}\,
{\displaystyle \int_{S^{n-1}} (M_+\!-\!M_-)\, d\xi}
\ + \ |\nabla\Phi_\lambda|\Big\vert_{S^{n-1}}
\bigl(\|M_+\|^2 +\|M_-\|^2\bigr)\\
\notag 
&=: V(M_+,M_-)\,,
\end{eqnarray}
where $\|\cdot\|$ denotes the $L^2$-norm on $S^{n-1}$.
For the second variation, we assume that
$\lambda>1$ and approximate
\begin{align}
\label{eq:W}
\W(A) &\approx  \ \int_{S^{n-1}}\int_{S^{n-1}} M(\xi) M(\eta)
\phi_\lambda(\xi-\eta) \, d\xi d\eta =: \ W(M)\,,
\end{align}
where $M=M_+-M_-$. Thus $\delta(A)\gtrapprox V(M_+,M_-) -W(M_+\!-\!M_-)$.

\subsection*{The toy model}
Consider minimizing $V-W$ 
over pairs of nonnegative functions 
$M_+, M_-$ on the sphere, subject to the constraints in 
Eq.~\eqref{eq:constraints}. 
We diagonalize the quadratic form
$W$ by expanding $M_+$ and $M_-$
in spherical harmonics.  The eigenvalues of $W$ are 
the Funk-Hecke multipliers associated
with $\phi_\lambda$, which form a strictly decreasing
sequence $(\beta_k)_{k\ge 0}$
satisfying $\lim \beta_k=0$, see Eq.~\eqref{eq:FH}.
Since the constraints in Eq.~\eqref{eq:constraints}
require $M=M_+-M_-$ to be orthogonal to the spherical harmonics of
degree zero and one,
$$ 
W(M_+\!-\!M_-) \le \beta_2 \|M_+\!-\!M_-\|^2\le 
\beta_2 \bigl(\|M_+\|^2+\|M_-\|^2\bigr)\,,
$$
see Lemma~\ref{lem:toy}.  We have used that
$M_+$ and $M_-$ are nonnegative to
drop the mixed term.
Returning to Eq.~\eqref{eq:V}, note that
the first integral vanishes since $|A|=|B^n|$,
and the coefficient of the second term is given by
$|\nabla\Phi_\lambda|\Big\vert_{S^{n-1}} = \beta_1  >  \beta_2$,
see Eq.~\eqref{eq:coeff1}. Therefore
$$V(M_+,M_-)\ge \beta_1\bigl(\|M_+\|^2+\|M_-\|^2\bigr)\,.
$$
We conclude that
\begin{equation}
\label{eq:toy}
V(M_+,M_-)-W(M_+\!-\!M_-) 
\ge  (\beta_1-\beta_2) \bigl(\|M_+\|^2+ \|M_-\|^2\bigr)\,.
\end{equation}
Since $\|M_+\|^2+\|M_-\|^2\ge \frac{\alpha^2}{2|S^{n-1}|}$
by Schwarz' inequality and Eq.~\eqref{eq:constraints},
this establishes the conclusion of the theorem
in the toy model. 
Note that the toy model makes sense only for $\lambda>1$,
where the Riesz potential is locally integrable on 
$S^{n-1}$.

\bigskip 
The main part of the paper is dedicated to justifying 
the approximations of $\V$ and $\W$ by the
spherical integrals $V$ and $W$ under the 
assumption of~\eqref{eq:ring}.
This is done in Sections~\ref{sec:V} and~\ref{sec:W}.
For $\V$, we rearrange $A\setminus B^n$ and $B^n\setminus A$ into
thin ring-shaped sets adjacent to the unit sphere,
and then Taylor expand $\Phi_\lambda(r\xi)$
about $r=1$.  For $\W$, we cannot simply 
expand $\phi_\lambda(r\xi-s\eta)$
about $r=s=1$ because of the singularity at $\xi=\eta$.
Instead, we represent $\Chi_A(r\xi)$ in spherical harmonics 
and control $\W(A)-W(M)$ as $\eps\to 0$
by bounds on the multiplication operator
associated with $\phi_\lambda$, see Lemma~\ref{lem:FH-r}.
In Section~\ref{sec:reduce}, we establish
the geometric conditions in Eqs.~\eqref{eq:ring} 
and ~\eqref{eq:constraints}.
Finally, Theorem~\ref{thm:main} is proved in Section~\ref{sec:main}.

\section{Some deformations of the ball}
\label{sec:deform} Setting aside for the moment the
validation of the toy model,
we illustrate its use by computing some examples.
The first two examples determine the coefficients in Eq.~\eqref{eq:V} 
in terms of the multipliers $\beta_0$ and $\beta_1$.

\begin{example}[Translation]
\label{eq:translate}
Let $A=-\eps e_n +B^n$.  Clearly, $\delta(A)=0$, and thus $\V(A)=\W(A)$.
The functions $M_+$  and $M_-$
from Eq.~\eqref{eq:def-M} are the positive and negative
parts of $ M(\xi) = \eps e_n\cdot \xi + O(\eps^2)$.
For the first variation, Eq.~\eqref{eq:V} yields
$$
\V(A) \approx V(M_+,M_-)
\approx \frac{1}{n} \, 
|\nabla \Phi_\lambda|\,\Big\vert_{S^{n-1}}\, |S^{n-1}|\,
\eps^2 \,,
$$
and for the second variation, Eq.~\eqref{eq:W} yields
$$
\W(A) \approx W(M) \approx
\frac{1}{n}\beta_1 \,|S^{n-1}|\,\eps^2\,,
$$
since $M$ is to leading order a spherical harmonic of degree one.
The errors are of order $O(\eps^3)$ and $o(\eps^2)$,
respectively, as $\eps\to 0$, see
Propositions~\ref{prop:V} and~\ref{prop:W}.
Comparing coefficients, we conclude that
\begin{equation}
\label{eq:coeff1}
|\nabla\Phi_\lambda|\Big\vert_{S^{n-1}}=\beta_1\,.
\end{equation}
\end{example}

\begin{example}[Dilation]
\label{eq:dilate}
Let $A=e^\eps B^n$, where $\eps>0$ is small.  By scaling, $|A|-|B^n| = (e^{n\eps}-1)|B^n|$, and
$$
\V(A)-\W(A)= \E_\lambda(B^n)-\E_\lambda(A) =
-(e^{(n+\lambda)\eps}\!-\!1)\E_\lambda(B^n)\,.
$$
Obviously $M_+ \equiv \eps + O(\eps^2)$ while $M_-$ vanishes
(the roles would be reversed if $\eps<0$).
For the first variation, Eq.~\eqref{eq:V} yields
\begin{align*}
\V(A) & \approx 2
\Phi_\lambda\Big\vert_{S^{n-1}}\, \bigl(|B^n|-|A|\bigr)
+ |\nabla\Phi_\lambda|\Big\vert_{S^{n-1}}\,
\bigl(\|M_+\|^2+\|M_-\|^2\bigr)\\
&\approx  -
\Phi_\lambda\Big\vert_{S^{n-1}}\,|S^{n-1}|\, (2\eps +\eps^2n) 
   + \beta_1 \,|S^{n-1}|\, \eps^2\,.
\end{align*} 
In the first term, we have expanded to second order
in $\eps$ and replaced $n|B^n|$ with $|S^{n-1}|$.
In the second term, we have inserted the value of the gradient
from Eq.~\eqref{eq:coeff1}.
For the second variation, Eq.~\eqref{eq:W} yields
$$
\W(A) \approx W(M) \approx \beta_0 \,|S^{n-1}|\, \eps^2\,,
$$
since $M=M_+$ is constant on $S^{n-1}$. 
As in the previous example, all errors are of order $o(\eps^2)$,
see Propositions~\ref{prop:V} and~\ref{prop:W}.
Comparing the coefficients of $\eps$ and $\eps^2$,
we conclude that
\begin{equation}
\label{eq:coeff0}
\Phi_\lambda\Big\vert_{S^{n-1}}
= \frac{\beta_0 - \beta_1}{\lambda}\,,\qquad
\E_\lambda(B^n)= 2\frac{\beta_0 - \beta_1}{\lambda(n+\lambda)}
\,|S^{n-1}|\,.
\end{equation}
\end{example}

The next example demonstrates that the quadratic 
scaling in Theorem~\ref{thm:main} is sharp.  

\begin{example}[Ring]
\label{ex:ring} 
Take $M_+= M_-\equiv c\eps$, 
where $c= \left(2|S^{n-1}|\right)^{-\frac12}$
so that $\|M_+\|^2 +\|M_-\|^2=\eps^2$.
Let $A\!\setminus\!B^n$ and $B^n\!\setminus \!A$ be narrow annuli
of the appropriate width that meet at the unit sphere.
For the asymmetry $\alpha=\int_{S^{n-1}}(M_+ +M_-)\, d\xi$, we
find $\alpha^2(A) \approx 2\,|S^{n-1}|\, \eps^2$,
since  the functions $M_\pm$ are constant.
Further, $\W(A)\approx 0$ since $M\equiv 0$,
and by Eqs.~\eqref{eq:V} and~\eqref{eq:coeff1}
$$
\delta(A)\approx \V(A)\approx \beta_1 \eps^2
\approx \frac{\beta_1}{2\,|S^{n-1}|} \alpha^2(A)
$$
up to errors of order $o(\eps^2)$.
\end{example}

In the final example, the ball is squeezed 
into an approximate ellipsoid. 

\begin{example}[Squeeze]
\label{ex:squeeze}
Let 
$M(\xi)=c \eps (n(\xi\cdot e_n)^2-1)$,
where $c= \left(\frac{4n}{n+2}|S^{n-1}|\right)^{-\frac12}$
is chosen
so that $\|M\|=\eps$, and let $M_\pm$ its positive and negative
parts. Define the corresponding
set $A$ by its radial function, $R(\xi)$.
Since $M$ is a spherical harmonic of degree  
two, $\int_{S^{n-1}}M\, d\xi=0$,
and the deformation is volume-preserving.
By Schwarz' inequality, the asymmetry satisfies
$\alpha^2(A) = \left(\int_{S^{n-1}}|M|\, d\xi\right)
\le |S^{n-1}|\,\eps^2$. 
Eqs.~\eqref{eq:V} and~\eqref{eq:W} yield $ \V(A) \approx
\beta_1\eps^2$ and $\W(A)\approx \beta_2\eps^2$,
with errors of order $o(\eps^2)$. Therefore, 
the deficit is given by
$$\delta(A) \approx (\beta_1-\beta_2)\eps^2
\ge \frac{\beta_1-\beta_2}{|S^{n-1}|}\alpha^2(A)\,.
$$
\end{example}

\section{First variation}
\label{sec:V}

In this section, we estimate $\V(A)$ for sets 
that satisfy Eq.~\eqref{eq:ring}, and justify the approximation
in Eq.~\eqref{eq:V}. Fix $\lambda\in (0,n)$,
and recall that
$\Phi_\lambda=\phi_\lambda * \Chi_{B^n}$
denotes the potential of the unit ball.

\begin{prop}[Projection to the sphere]
\label{prop:V} 
For $A\subset \RR^n$, define $M_+$ and $M_-$ by Eq.~\eqref{eq:def-M}. 
Let $\V(A)$ be the first variation in Eq.~\eqref{eq:def-V},
and let $V(M_+,M_-)$ be the spherical integral on the right hand side
of Eq~\eqref{eq:V}.
Then
\begin{align}
\label{eq:VV}
\V(A) \ \ge\  V(M_+,M_-) + O(\eps)\bigl(\|M_+\|^2+ \|M_-\|^2\bigr) 
\end{align}
as $\eps\to 0$ for every $A$ satisfying Eq.~\eqref{eq:ring}.
\end{prop}

\begin{proof} By definition,
\begin{align}
\label{eq:V-proof-1}
\V(A) = 
2 \left(
\int_{B^n\setminus A} \Phi_\lambda(x)\, dx
- \int_{A\setminus B^n} \Phi_\lambda(x)\, dx 
\right)\,.
\end{align}
In polar coordinates, the second integral on
the right hand side takes the form
$$
\int_{A\setminus B^n} \Phi_\lambda(x)\, dx =
\int_{S^{n-1}} \int_1^\infty 
\Chi_A(r \xi)\Phi_\lambda(r\xi)\,r^{n-1}\, dr d\xi \,.
$$
Consider a rearrangement of $A\setminus B^n$
where mass is moved inwards along rays 
in such a way that for each $\xi\in S^{n-1}$ the
value of $M_+(\xi)$ is preserved and the intersection with the ray
through $\xi$ becomes an interval $[1,1\!+\!R_+(\xi)]$.  
By construction, $0\le R_+\le e^\eps-1$, and 
$M_+ = \frac 1n ((1\!+ \!R_+)^n-1)$.
Since $\Phi_\lambda$ is radially decreasing,
this rearrangement increases the value of the integral,
$$
\int_{A\setminus B^n} \Phi_\lambda(x)\, dx
\le \int_{S^{n-1}} \int_1^{1+R_+(\xi)}
\Phi_\lambda(r\xi)\, r^{n-1}\, dr\, d\xi\,.
$$
From the Taylor expansion of $\Phi_\lambda(r\xi)$ about
$r=1$, we obtain for the inner integral
\begin{align*}
\int_1^{1+R_+(\xi)} \Phi_\lambda(r\xi)\, r^{n-1}\, dr
&= \Phi_\lambda(\xi) M_+(\xi)
+\bigl( \xi\!\cdot\,\nabla\Phi_\lambda(\xi)+ O(\eps)\bigr)
 \int_1^{1+R_+(\xi)}
(r-1)r^{n-1}\, dr \\
&= \Phi_\lambda(\xi) \, M_+(\xi)
- \frac12 \left(|\nabla\Phi_\lambda(\xi)| + O(\eps\right)
|M_+(\xi)|^2 \,.
\end{align*}
In the last step, we have integrated explicitly
and used that $R_+=(1+O(\eps))M_+$.
The quadratic term appears with a negative sign, because
$\Phi_\lambda$ is radially decreasing.
All error estimates hold uniformly 
for $0\le R_+\le e^\eps-1$ and $\xi\in S^{n-1}$.
Integration over $S^{n-1}$ yields
\begin{equation}
\label{eq:V-proof+}
\int_{A\setminus B^n} \Phi_\lambda(x)\, dx
\ \le \ \left(\Phi_\lambda \Big\vert_{S^{n-1}}\right)
|A\!\setminus \!B^n|
\ - \ \frac12 
\left(|\nabla\Phi_\lambda|\, \Big\vert_{S^{n-1}}+ O(\eps)\right)
\|M_+\|^2 \,.
\end{equation}

Similarly, the first integral on the right hand side of
Eq.~\eqref{eq:V-proof-1} takes the form
$$
\int_{B^n\setminus A} \Phi_\lambda(x)\, dx =
\int_{S^{n-1}} \int_0^1 
(1\!-\!\Chi_A(r \xi))\Phi_\lambda(r\xi)\,r^{n-1}\, dr d\xi \,.
$$
This decreases under the rearrangement of $A\cap B^n$ that moves
mass inwards such that the intersection with each ray defined
by a direction $\xi\in S^{n-1}$ 
is replaced with the interval $[0, 1\!-\!R_-(\xi)]$,
where $0\le R_-\le 1-e^{-\eps}$, and
$M_- = \frac 1n (1-(1\!-\! R_-)^n)$,
$$
\int_{B^n\setminus A} \Phi_\lambda(x)\, dx
\ge \int_{S^{n-1}} \int_{1-R_-(\xi)}^1
\Phi_\lambda(r\xi)\, r^{n-1}\, dr\, d\xi\,.
$$
As above, we use the Taylor expansion of $\Phi_\lambda(r\xi)$ about
$r=1$ to obtain 
\begin{equation}
\label{eq:V-proof-2}
\int_{B^n\setminus A} \Phi_\lambda(x)\, dx
\ \ge \ \left( \Phi_\lambda \Big\vert_{S^{n-1}}\right)
|B^n\!\setminus \!A|
\ + \ \frac12
\left(|\nabla\Phi_\lambda|\, \Big\vert_{S^{n-1}}+O(\eps)\right)
\|M_-\|^2 \,.
\end{equation}
This time, the quadratic term appears with the positive sign.
The proof is completed by subtracting Eq.~\eqref{eq:V-proof+}
from Eq.~\eqref{eq:V-proof-2}.
\end{proof}

\section{Second variation}
\label{sec:W}

In this section, we estimate $\W(A)$ for sets satisfying
Eq.~\eqref{eq:ring},
and  justify the approximation in Eq.~\eqref{eq:W}.
Fix $\lambda\in (1,n)$.  

\begin{prop}[Projection to the sphere]
\label{prop:W} 
For $A\subset \RR^n$, define $M_+$ and $M_-$ by Eq.~\eqref{eq:def-M}. 
Let $\W(A)$ be the first variation in Eq.~\eqref{eq:def-W},
and let $W(M_+-M_-)$ be the spherical integral on the right hand side
of Eq~\eqref{eq:W}.
Then
\begin{equation}
\label{eq:WW}
\W(A)= W(M_+-M_-) + o(1)\bigl(\|M_+\|^2+ \|M_-\|^2\bigr)
\end{equation}
as $\eps\to 0$ for every A satisfying Eq. (2.2).
\end{prop}

As in the previous section, we work in polar coordinates.

\begin{lem}[Error term]
\label{lem:W}
For $\lambda\in (1,n)$ and $\eps>0$, set
$$\psi_{\lambda,\eps}(\xi,\eta):= e^{(n-\lambda)\eps}\phi_\lambda(\xi-\eta)
- e^{-2(n-\lambda)\eps} \phi_\lambda (e^{-2\eps}\xi-\eta)\,,
$$
and let $L_{\lambda,\eps}$ be the linear operator
defined by
\begin{equation}
\label{eq:L}
(L_{\lambda,\eps}M)(\xi):= \int_{S^{n-1}} \psi_{\lambda,\eps}(\xi,\eta)
M(\eta)\, d\eta
\end{equation}
for $M\in L^2$ and $\xi\in S^{n-1}$.
If $A\subset R^n$ satisfies Eq.~\eqref{eq:ring}, then
$$
|\W(A)-W(M)| \le  2 \|L_{\lambda,\eps}\|_{L^2\to L^2}\,
\bigl(\|M_+\|+\|M_-\|)^2\,.
$$
\end{lem}

\begin{proof} 
By decomposing $\Chi_{B^n}-\Chi_A$ into
its positive and negative parts and using Schwarz' inequality, 
it suffices to prove that
\begin{equation}
\label{eq:DL}
\begin{split}
D \ &:=\Bigl|\int_{A_i}\int_{A_j} \phi_\lambda(x-y)\, dxdy 
- \int_{S^{n-1}} \int_{S^{n-1}} M_i(\xi)M_j(\eta)\phi_\lambda(\xi-\eta)
\, d\xi d\eta\Bigr|\\
&\qquad \le \|L_{\lambda,\eps}\|_{L^2\to L^2} \,\|M_i\|\, \|M_j\|
\end{split}
\end{equation}
for $i,j\in\{+, -\}$, where $A_+=A\setminus B^n$, and 
$A_-=B^n\setminus A$.
In polar coordinates,
$$
D= \int_{S^{n-1}} \int_{S^{n-1}}
\left(
\int_0^\infty\!\!\int_0^\infty
\Chi_{A_i}(r\xi)\Chi_{A_j}(s\eta) 
\bigl(\phi_\lambda(r\xi\!-\!s\eta)-\phi_\lambda(\xi\!-\!\eta)\bigr)\,
(rs)^{n-1}drds\right)d\xi d\eta\,.
$$

\smallskip We claim that for $r,s\in [e^{-\eps}, e^\eps]$,
\begin{equation}
\label{eq:bounds-phi}
e^{-(n-\lambda)\eps}\phi_\lambda(\xi,\eta)\le\phi_\lambda(r\xi-s\eta)
\le e^{2(n-\lambda)\eps}\phi_\lambda(e^{-2\eps}\xi,\eta)\,.
\end{equation}
The upper bound holds since
$$ |r\xi-s\eta|^2 = (r-s)^2
+ rs |\xi-\eta|^2 \ge e^{-2\eps} |\xi-\eta|^2\,, 
$$ 
and the lower bound follows from
$$
|r\xi-s\eta|^2 = rs \left( \tfrac{r}{s}+ \tfrac{s}{r}
-2\xi\cdot\eta \right) 
\le e^{2\eps}\left( e^{2\eps}+e^{-2\eps}
-2\xi\cdot\eta\right) = e^{4\eps}
\left|e^{-2\eps}\xi-\eta\right|^2\,.
$$
In the middle step,
we have used the convexity of $t\mapsto t+ t^{-1}$ 
to replace $\frac r s $ with $e^{-2\eps}$.

Since $\phi_\lambda(\xi-\eta)$ lies between the upper
and lower bound, Eq.~\eqref{eq:bounds-phi} implies that
$$
\sup_{r,s\in [e^{-\eps}, e^\eps]}\ 
|\phi_\lambda(r\xi-s\eta)-\phi_\lambda(\xi-\eta)|\le 
\psi_{\lambda,\eps}(\xi,\eta)\,.
$$ 
Therefore, for each $\xi,\eta\in S^{n-1}$,
the inner integral satisfies
$$
\left|
\int_0^\infty\!\!\int_0^\infty
\Chi_{A_i}(r\xi)\Chi_{A_j}(s\eta) 
\bigl(\phi_\lambda(r\xi\!-\!s\eta)-\phi_\lambda(\xi\!-\!\eta)\bigr)\,
(rs)^{n-1}drds\right| 
\le M_i(\xi)M_j(\eta) \psi_{\lambda,\eps}(\xi,\eta)\,.
$$
Integration over the spherical variables yields Eq.~\eqref{eq:DL}.
\end{proof}

The remainder of the section is dedicated to 
computing $W(M)$ and bounding the norm of $L_{\lambda,\eps}$.
Since $W$ is rotation-invariant, we 
will expand $M$ in spherical harmonics.  We briefly
summarize the pertinent facts about spherical harmonics,
following the conventions in~\cite{AAR-1999}.

By definition, a {\em spherical harmonic}
is the restriction of a harmonic homogeneous polynomial
to the unit sphere. Explicitly,
$Y_k$ is a spherical harmonic of degree $k\ge 0$
on $S^{n-1}$, if and only if the 
function $P_k(x) = |x|^kY_k\bigl(\frac{x}{|x|}\bigr)$
is a harmonic homogeneous polynomial of degree $k$
on $\RR^n$.

The spherical harmonics simultaneously diagonalize all 
rotation-invariant linear operators on $L^2(S^{n-1})$.
We first consider the functional $W$.
By the Funk-Hecke formula,
$$
\int_{S^{n-1}} Y_k(\eta)\phi_\lambda(\xi-\eta) \, d\eta = 
\beta_k Y_k(\xi)\,,\quad \xi\in S^{n-1}
$$
for every spherical harmonic $Y_k$ of degree $k$ on
on $S^{n-1}$,
with multipliers $\beta_k$ that depend only
on $k$, $n$, and $\lambda$. 
In terms of these multipliers,
\begin{equation}
\label{eq:W-beta}
W(M)=\sum_{k\ge 0} \beta_k \|Y_k\|^2\,,\qquad M\in L^2(S^{n-1})\,,
\end{equation}
where $M=\sum_{k\ge 0} Y_k$ is the expansion of
$M$ in spherical harmonics.

Let us sketch the computation of the Funk-Hecke multipliers.
Since the spherical harmonics of degree
$k=0$ are the constant functions,
\begin{equation}
\label{eq:FH0}
\beta_0=\frac{1}{c_\lambda}
\int_{S^{n-1}} |e_n-\eta|^{-n+\lambda}\, d\eta>0\,.
\end{equation}
For $k\ge 1$ in dimension $n\ge 3$, we apply the Funk-Hecke formula
to the zonal harmonic 
$Z_k(\eta) = C_k(e_n\cdot\eta)$, where $C_k$
is a Gegenbauer polynomial of order $\nu=\frac{n-2}{2}$, 
and evaluate the integral at $\xi=e_n$.
By definition, $C_k(t)$ is a polynomial of
degree $k$, given by the $k$-th Taylor coefficient
of $(1+r^2-2rt)^{\frac{n-2}{2}}$ at $r=0$. In particular,
$Z_k(e_n)= C_k(1) = \frac{\Gamma(n-2+k)} {k!\, \Gamma(n-2)}$.
Note that $C_k$ is even when $k$ is even, and odd otherwise.
We compute 
\begin{align*}
\beta_k &= \frac{1}{c_\lambda Z_k(e_n)}
\int_{S^{n-1}} Z_k(\eta) |e_n-\eta|^{-n+\lambda} \,d\eta\\
&= \frac {|S^{n-2}| }{C_k(1)c_\lambda} 
\int_0^\pi C_k(\cos\theta) \, (2-2\cos\theta)^{-\frac{(n-\lambda)}{2}}
(\sin\theta)^{n-2}\, d\theta\\
&= \frac{|S^{n-2}|} 
{c_\lambda 2^{\frac{n-\lambda}{2}}C_k(1) }
\, \int_{-1}^1 C_k(-t) (1+t)^{\frac{\lambda-3}{2}} (1-t)^{\frac{n-3}{2}} 
\, dt\\
&= (-1)^k \frac{c_{n,\lambda}}
{\Gamma(\frac{\lambda-n+2}{2}-k )
\Gamma(\frac{\lambda+n-2}{2}+k)}
\end{align*}
for some constant $c_{n,\lambda}>0$ that does
not depend on $k$. 
Here, $\theta$ is the angle from the north pole,
$t=-\cos\theta$, and we have expanded
$\sin^2 \theta = (1-t)(1+t)$.
In the last line, we have turned to the 
tables of Erdelyi et al~\cite{EMOT-1954}
and applied Eq. (3) on  p.~280
(with parameter values $n=k$, $\nu=\frac{n-2}{2}$,
$\beta=\frac{\lambda-3}{2}$).
The functional equation for Gamma yields the recursion relation
\begin{equation}
\label{eq:FH}
\beta_{k+1} = \frac{n-\lambda+2k}
{n+\lambda+2k-2}\,\beta_k\,, \qquad k\ge 0, n\ge 3\,.
\end{equation}
Clearly, the sequence $(\beta_k)$ is positive, 
decreasing, and converges to zero. 
The order of decay is $O(k^{-(\lambda-1)})$.

In dimension $n=2$, the zonal harmonics
are given by 
$Z_k=\cos(k\theta)=T_k(\cos\theta)$, 
where $T_k$ is a Chebyshev polynomial.
Here, we use 
Eq. (1) on  p.~271 (with parameter values
$n=k$, $a=\frac{\lambda-3}{2}$) to derive 
the slightly different recursion relation
$$
\beta_{k+1} = \frac{(2k+2)(2k+2 -\lambda)}
{(2k+1)(2k+\lambda)}\,\beta_k\,,\qquad k\ge 0, n=2\,.
$$
The sequence is decreasing for $k\ge 1$; if $\lambda>\frac43$ then
also $\beta_1<\beta_0$.

\begin{prop} [Bound on $W$]
\label{lem:toy}
Let $\lambda\in (1,n)$.
If $F$ is a square-integrable function on
$S^{n-1}$ that satisfies
$$
\int_{S^{n-1}} F(\xi)\, d\xi=0\,,\qquad
\int_{S^{n-1}} \xi F(\xi)\, d\xi =0\,,
$$
then $W(F)\le \beta_2\|F\|^2$.
\end{prop}

\begin{proof} 
Expand in spherical harmonics $F=\sum_{k\ge 0} Y_k$, 
and apply the Funk-Hecke formula\,,
$$
W(F) \ = \ \sum_{k\ge 2} \beta_k \|Y_k\|^2 
\ \le \ \beta_2 \sum_{k\ge 2} \|Y_k\|^2 
\ =\  \beta_2 \|F\|^2\,.
$$
Here, the first and last identities hold
since $Y_0=Y_1=0$ by assumption, and the middle
inequality follows since $\beta_k<\beta_2$ for all $k>2$.
\end{proof}

To bound the norm of $L_{\lambda,\eps}$, we also
need estimates on the multipliers $b_k(r)$ associated with
the kernels $\phi_\lambda(r\xi-\eta)$.  
The multipliers are defined by the property that
\begin{equation}
\label{eq:FH-r}
\int_{S^{n-1}} Y_k(\eta)\phi_\lambda(r\xi-\eta)\, d\eta = 
b_k(r)Y_k(\xi)\  \qquad \xi\in S^{n-1}
\end{equation}
for every spherical harmonic $Y_k$. 
We note in passing that, for any pair of spheres of different
radii $0<r\le s$, the multipliers are given by
$s^{-(n-2)}b_k\bigl(\frac{r}{s}\bigr)$.
When $\lambda=2$, the multipliers are easily computed:

\begin{lem} [Funk-Hecke multipliers for the Newton potential]
\label{lem:FH-N}
For $n\ge 3$, $\lambda=2$,
$$
b_k(r) = \frac{r^k}{n+k-2} \,,\qquad k\ge 0, r\in [0,1]\,.
$$
\end{lem}

\begin{proof} 
At $r=1$, Eq.~\eqref{eq:FH0} and the recursion 
in Eq.~\eqref{eq:FH}
yield $\beta_k= \frac1{n+k-2}$.

Let $Y_k$ be a spherical harmonic of order $k\ge 0$.
Since $\phi_2$ is the fundamental solution of the Laplacian,
the function
$$
u(x) := b_k(|x|)Y_k\bigl(\tfrac{x}{|x|}\bigr)
=\int_{S^{n-1}} Y_k(\eta)\phi_2(x-\eta)\, d\eta\,,
\quad (x\in B^n)\, 
$$
is harmonic on the unit ball, with boundary values
on $S^{n-1}$ given by $\beta_kY_k$.
But the unique harmonic extension of 
the spherical harmonic $\beta_kY_k$ to the unit
ball is the homogeneous harmonic polynomial of degree $k$ 
that defines it. Therefore $u(r\xi)=\beta_k r^k Y_k(\xi)$,
proving the claim.
\end{proof}

As a consequence, for $\lambda=2$ 
Proposition~\ref{prop:W} holds with an explicit 
error estimate of order~$O(\eps)$:

\begin{proof}[Proof of Proposition~\ref{prop:W} for the 
Newton potential] Let $n\ge 3$ and $\lambda=2$. 
By Lemma~\ref{lem:W} and Schwarz' inequality,
$$| \W(A)-W(M)|\le 2 \|L_{2,\eps}\|_{L^2\to L^2} 
\bigl(\|M_+\|^2+\|M_-\|^2\bigr)\,.
$$
By the definition of $L_{2,\eps}$
and Lemma~\ref{lem:FH-N},
the operator is represented by the  multipliers
$$
e^{(n-2)\eps}\beta_k - e^{-2(n-2)\eps}b_k(e^{-2\eps})
= \frac{e^{(n-2)\eps}-e^{-2(n+k-2)\eps}}{n+k-2}\,,
\qquad k\ge 0\,.
$$
Its operator norm is the norm of the sequence of
multipliers in $\ell^\infty$,
$$
\|L_\eps\|_{L^2\to L^2}
=\sup_{k\ge 0} 
\frac{e^{(n-2)\eps}-e^{-2(n+k-2)\eps}}{n+k-2}
= \frac{e^{(n-2)\eps}-e^{-2(n-2)\eps}}{n-2}\,.
$$
In particular, $\|L_\eps\|_{L^2\to L^2}=O(\eps)$ as $\eps\to 0$.
\end{proof}

\bigskip For $\lambda\ne 2$, we have the following
estimate.

\begin{lem}[Funk-Hecke multipliers for radius $r<1$]
\label{lem:FH-r}
Let $\lambda\in (1,n)$. For $k\ge 0$, define 
$b_k(r)$ by Eq.~\eqref{eq:FH-r}.
Then $b_k$ is continuous on $[0,1]$, and
$$
b_k(r)\le \left(\frac{2}{1+r^2}\right)^{\frac{n-\lambda}{2}}
\beta_k\,,\qquad k\ge 0\,,\ r\in [0,1]\,,
$$
where $\beta_k$ is as in Eqs.~\eqref{eq:FH0} and~\eqref{eq:FH}.
Equality holds for $r=1$.
\end{lem}

\begin{proof} Fix $\lambda\in (1,n)$ and $k\ge 0$.
We evaluate Eq.~\eqref{eq:FH-r}
with $Y_k=Z_k$ (the zonal harmonic) at $\xi=e_n$,
\begin{align}
\label{eq:bk-r-proof}
b_k(r)&=
\frac{1}{c_\lambda Z_k(e_n)}
\int_{S^{n-1}}
Z_k(\eta) |re_n-\eta|^{-(n-\lambda)}\, d\eta\,.
\end{align}
For $\eta\ne e_n$, the Riesz potential is bounded by
\begin{align*}
\phi_\lambda(re_n-\eta) = 
\frac{1}{r^{\frac{n-\lambda}{2}}c_\lambda}
\left(\frac{1+r^2}{r} -2e_n\cdot \eta\right)^{-\frac{n-\lambda}{2}}
&\le r^{-\frac{n-\lambda}{2}}\phi_\lambda(e_n-\eta)\,.
\end{align*}
Since each $Z_k$ is bounded on $S^{n-1}$,
the Dominated Convergence Theorem
implies that $b_k(r)$ is continuous on $[0,1]$,
and that $b_k(1)=\beta_k$.


To obtain the upper bound on $b_k(r)$,
we use the binomial expansion,
\begin{equation}
\label{eq:eq:expand-phi}
(1+r^2)^{{\frac{n-\lambda}{2}}}|re_n -\eta|^{-(n-\lambda)}
= \left( 1-\tfrac{2r}{1+r^2} e_n\cdot\eta\right)^{-\frac{n-\lambda}{2}}
=\sum_{\ell=0}^\infty a_\ell t^\ell\,,
\end{equation}
where $t=\frac{2r}{1+r^2}e_n\cdot \eta$.
Since the exponent is negative,
the coefficients in the binomial series are positive.
We then integrate the series against the zonal
harmonic $Z_k$. Note that  
$$
\int_{ S^{n-1}} Z_k(\eta)(e_n\cdot\eta)^\ell \, d\eta=0\,,
\qquad \text{if}\ \ell<k\ \text{or}\ \ell-k\ \text{is odd}\,,
$$
because $Z_k$ is orthogonal to all polynomials
of order less than $k$ and contains only monomials of the
same parity as $k$. 
For $\ell=k+2j$, the integrals can be evaluated
exactly in terms of Gamma functions.
In dimension
$n \geq 3$, we apply Eq. (2) on p. 280 of~\cite{EMOT-1954}; in dimension
$n=2$ we interpret $\eta\in S^1$ as a complex variable,
write $Z_k(\eta)={\rm Re}\, \eta^k$ and
$\eta\cdot e_2=\frac12(\eta+\eta^{-1})$, and apply Cauchy's formula.
In either case,
$$
d_{k,j}:=\int_{S^{n-1}} Z_k(\eta)(e_n\cdot \eta)^{k+2j} \, d\eta >0\,.
$$

Since all coefficients in the series
$$
(1+r^2)^{\frac{n-\lambda}{2}}
b_k(r)= \frac{1}{ c_\lambda Z_k(e_n)}
\sum_{j=0}^\infty a_{k+2j}d_{k,j} \left(\frac{2r}{1+r^2}\right)^{k+2j} 
$$
are positive, it defines an increasing function of $r$ on $[0,1]$.  
The proof is completed by comparing with the value
at $r=1$.
\end{proof}

The proof of Lemma~\ref{lem:FH-r}
shows that $b_k(r)=O(r^k)$ as $r\to 0$.
We suspect that $b_k$ itself may be increasing on $[0,1]$
but could not find a reference.

\begin{proof}[Proof of Proposition~\ref{prop:W}]
Let $n\ge 2$ and fix $\lambda\in (1,n)$.
By Lemma~\ref{lem:W} and Schwarz' inequality,
$$
|\W(A)-W(M)| \le
2\|L_{\lambda,\eps}\|_{L^2\to L^2} \,\bigl(\|M_+ \|^2+ \|M_-\|^2)\,.
$$
By definition, the operator $L_{\lambda,\eps}$
is represented by the multipliers
$ e^{(n-\lambda)\eps}\beta_k - e^{-2(n-\lambda)\eps}b_k(e^{-2\eps})>0$.
Its norm is bounded by
$$
\|L_{\lambda,\eps}\|_{L^2\to L^2}
=\sup_{k\ge 0}
\left\{
e^{(n-\lambda)\eps}\beta_k - e^{-2(n-\lambda)\eps}b_k(e^{-2\eps})\right\}\,.
$$
By Lemma~\ref{lem:FH-r},
$$
e^{-2(n-\lambda)\eps}b_k(e^{-2\eps})
\le \left(\frac{2e^{-2(n-\lambda)\eps}}
{1+e^{-4\eps}}\right)^\frac{n-\lambda}{2}
\beta_k \le \beta_k\,.
$$
In particular, the multipliers are positive.
Since $\beta_k$ is decreasing, we have for any $K>0$
$$
\|L_{\lambda,\eps}\|
\le \max
\left\{\max_{k<K} (b_k(r)-\beta_k),e^{(n-\lambda)\eps}\beta_K\right\}\,.
$$
Since $\lim_{\eps\to 0}b_k=\beta_k$ for each $k$ by Lemma~\ref{lem:FH-r}, 
it follows that
$$
\limsup_{\eps\to 0} \|L_{\lambda,\eps}\| \le \beta_K\,.
$$
We finally take $K\to\infty$ and recall that $\lim \beta_k=0$.
\end{proof}

\section{Geometric reduction}
\label{sec:reduce}

In this section, we reduce the proof of
Theorem~\ref{thm:main} to sets
that are squeezed between two balls,
as in Eq.~\eqref{eq:ring}, and satisfy
the constraints in Eq.~\eqref{eq:constraints}.
In the estimates, we use the notation 
$a\lesssim b$ (and equivalently $b\gtrsim a$)
to signify that $a\le  Cb$ for some constant $C$
that depends only on $n$ and $\lambda$.
We say that a subset $A\subset \RR^n$ is {\em scaled}
if $|A|=|B^n|$, or equivalently, $A^*=B^n$.
It is {\em centered} if $\alpha(A)=|A\Delta A^*|$.

\begin{prop}[Auxiliary properties]
\label{prop:reduce}
For every subset $A\subset\RR^n$ of finite
positive volume there exists a scaled subset $\tilde A \subset \RR^n$ 
and $\eps\lesssim(\alpha(A))^{\frac{\lambda}{n}}$
such that
\begin{align}
\tag{P1} &\delta(\tilde A) \le \delta(A)\,,\\
\tag{P2} &\alpha(\tilde A) = \alpha(A)\,,\\
\tag{P3} &e^{-\eps} B^n\subset \tilde A \subset e^\eps B^n\,,\\
\tag{P4} &\int_{\tilde A} \frac{x}{|x|}\, dx=0\,.
\end{align}
\end{prop}

In the proof of Proposition~\ref{prop:reduce},
we assume that $A$ is scaled and centered,
and move parts of its mass towards the origin  
to achieve (P1) and (P3). To ensure (P2),
we leave a narrow neighborhood of the unit 
circle unchanged. In the last step, 
a small translation yields (P4).

The first two lemmas will be used to establish Property (P1). 
As in Section~\ref{sec:outline}, we expand 
\begin{equation}
\label{eq:VVWW}
\delta(A)-\delta(\tilde A) =\V(A)-\V(\tilde A)
+ \W(A)-\W(\tilde A)\,,
\end{equation}
and separately estimate the contributions of $\V$ and $\W$.

\begin{lem} [Moving mass inwards, first variation]
\label{lem:reduce-V}
Let $\frac12 \le R_1<R_2 \le \frac32$, 
and let $A, \tilde A\subset\RR^n$ be scaled subsets with 
$ A\setminus \tilde A\subset (\RR^n\setminus R_2B^n)$
and $\tilde A \setminus A \subset R_1B^n$.
Then $$\V(A)-\V(\tilde A) \gtrsim (R_2-R_1)\, |\tilde A\Delta A|\,.
$$
\end{lem}
\begin{proof} 
By Eq.~\eqref{eq:def-V},
$$
\V(A)-\V(\tilde A)  = 2\int\bigl(\Chi_{\tilde A\setminus A}(x)
-\Chi_{A\setminus \tilde A}(x)\bigr)\Phi_\lambda\, dx\,.
$$
Since $\Phi_\lambda$ is smooth and radially 
decreasing, it follows that
\begin{align*}
\V(A)-\V(\tilde A)
&\ge 2 \Bigl(\Phi_\lambda \Big\vert_{|x|=R_2}\Bigr)
|\tilde A\setminus A|
- 2 \Bigl(\Phi_\lambda \Big\vert_{|x|=R_1}\Bigr)|A\setminus \tilde A|\\
&\ge \Bigl(\inf_{\frac12\le |x|\le \frac32}
|\nabla\Phi_\lambda|\Bigr)
(R_2-R_1)\, |\tilde A\Delta A|\,.
\end{align*}
In the last line, we have used that
$\frac12\le R_1<R_2\le \frac32$ and
$|\tilde A\setminus A|= |A\setminus \tilde A|=\frac12|\tilde A\Delta A|$.
\end{proof}

\begin{lem} [Moving mass inwards, second variation]
\label{lem:reduce-W} If $A, \tilde A\subset\RR^n$ 
satisfy $|A\Delta B^n|\le \alpha$, $|\tilde A\Delta B^n|\le \alpha$,
then 
$$|\W(A)-\W(\tilde A)|\lesssim \alpha^{\frac{\lambda}{n}}|\tilde A\Delta A|
\,.$$
\end{lem}

\begin{proof}
By Eq.~\eqref{eq:def-W},
\begin{align*}
\W(A)-\W(\tilde A)
&= \iint \bigl(\Chi_{\tilde A}(x)\!-\!\Chi_A(x)\bigr)
\bigl(\Chi_{\tilde A}(y)\!+\! \Chi_A(y)\!-\!2\Chi_B(y)\bigr)
\phi_\lambda(x,y)\, dxdy\,.
\end{align*}
The first factor in the integral is  supported  on $\tilde A\Delta A$,
where it takes the values $\pm 1$. The second
factor is supported on a set of measure at most $2\alpha$, where it
takes values in $\{0, \pm 1,\pm 2\}$. By the Riesz--Sobolev inequality,
\begin{align*}
\W(A)-\W(\tilde A)
&\le
\iint 
\bigl|\Chi_{\tilde A}(x)-\Chi_A(x)\bigr|\,
\big| \Chi_{\tilde A}(y)+ \Chi_A(y)-2\Chi_{B^n}(y)\bigr|
\phi_\lambda(x,y)\, dxdy\\
&\le \int_{\{|x|^n|B^n|< |\tilde A\Delta A|\}}
\int_{\{|y|^n|B^n|<2\alpha\}}
 2\phi_\lambda(x-y) \, dxdy\\
& \le 2 (2\alpha)^{\frac{\lambda}{n}} \Phi_\lambda(0)\,
|\tilde A\Delta A|\,.
\end{align*}
In the last line, we have rescaled the inner integral
to range over the unit ball, and then used that $\Phi_\lambda$
is radially decreasing.
\end{proof}

\bigskip
The next two lemmas will be used to establish Property (P2). 
They show that mass can be moved 
around without changing the asymmetry,
so long as a suitable neighborhood of the
unit circle is left untouched.

\begin{lem}[Symmetric difference of balls]
\label{lem:ball-translate}
For any $y\in \RR^n$, we have 
$$|B^n\Delta (y+B^n)|\ge \min\{|y|,2\}\,|B^n|\,.
$$
\end{lem}
\begin{proof} We may take
$y = (t, 0, \dots, 0)$ with $t\ge 0$. For $t\ge 2$, the balls
are disjoint and their symmetric
difference equals $2|B^n|$. For $t\in [0,2]$, let
$$
f(t):= |B^n\cap (te_1+B^n)| = 2\,\left|\left\{x\in B^n\ 
\vert\   x_1\ge \tfrac  t 2\right\}\right|\,.  
$$
Clearly, $f(0)=|B^n|$ and $f(2)=0$.
The derivative $f'(t)$
is given by a negative multiple of the cross-sectional 
area of $B^n$ at $x_1=\frac t 2$.
Since $f'(t)$ is increasing on $[0,2]$, $f$ is convex.
Therefore
$$
f(t) \le \left(1-\tfrac{t}{2}\right) f(0) + \tfrac t 2 f(2)
= \left(1-\tfrac{t}{2}\right) |B^n|\,.
$$
Since $|B^n\Delta (x+B^n)|= 2(|B^n|-f(|x|))$, this proves the claim.
\end{proof}


\begin{lem}[Preserving asymmetry while moving mass]
 \label{lem:rigid}
Let $\rho\in [0,1)$, and let $A\subset \RR^n$ be a 
scaled and centered subset
with asymmetry $\alpha(A)\le \frac{\rho}2|B^n|$.
If 
$$ \tilde A\cap (1+\rho)B^n = A\cap (1+\rho)B^n \,,
$$
then $\tilde A$ is centered and $\alpha(\tilde A)=\alpha(A)$.
The same conclusion holds if, instead,
$$
\tilde A\setminus (1-\rho)B^n = A\setminus (1-\rho)B^n\,.
$$
\end{lem}

\begin{proof} We need to show that 
$|\tilde A \Delta (y+B^n)|\ge |\tilde A \Delta B^n|=\alpha(A)$
for all $y\in\RR^n$.
If $\tilde A$ agrees with $A$ on $(1+\rho)B^n$, then 
$$
|\tilde A \Delta(y+B^n)|
= 2 |(y+B^n)\!\setminus \!\tilde A|=2 |(y+B^n)\!\setminus \!A|
\ge \alpha(A)\,, \qquad (|y|\le \rho)\,,
$$
since $y+B^n\subset (1+\rho)B^n$.
Similarly, if $\tilde A$ agrees with
$A$ on the complement of
$(1-\rho)B^n$, then
$$
|\tilde A \Delta(y+B^n)| = 2 |\tilde A \!\setminus \!(y+B^n)|=
2 |A\!\setminus \!(y+B^n)| \ge \alpha(A)\,, \qquad (|y|\le \rho)\,,
$$
since $y+B^n\supset (1-\rho)B^n$.
In either case, for $y=0$ we have
$|\tilde A\Delta B^n|= \alpha(A)$.
Moreover, by the 
reverse triangle inequality and Lemma~\ref{lem:ball-translate},
$$
|\tilde A \Delta (y+B^n)|
\ge |B^n\Delta (y+B^n)| - |\tilde A \Delta B^n|
 \ge \rho|B^n| - \alpha(A)\\
 > \alpha(A)\,,\qquad (|y|> \rho)\,,
$$
completing the proof.
\end{proof}

The next lemma will be used to establish Property (P4).

\begin{lem} [Median]
\label{lem:median} For $n\ge 2$, let $A\subset\RR^n$ be 
a bounded set of positive measure.
There is a unique point $x_0\in\RR^n$ such that
$$
\int_{x_0+A} \frac{y}{|y|}\, dy = 0\,.
$$
If $A$ is scaled and centered, with asymmetry $\alpha(A)=\alpha$,
then $|x_0| \lesssim \alpha(A)$.
\end{lem}

\begin{proof}  In dimension $n\ge 2$, the function 
$$
f(x)=\int_{x+A} |y|\, dy = \int_A |y-x|\, dy\,,\quad (x\in\RR^n)\,.
$$ 
is continuously differentiable and strictly
convex.  Since $f$ grows at infinity,
it has a unique minimizer, $x_0$, which is characterized
by the variational equation
$$
0\ = \ \nabla f(x_0)\ =\ \int_{x_0+A} \frac{y}{|y|}\, dy\,.
$$

Suppose that $A$ is scaled and centered,
with asymmetry $\alpha(A)=\alpha$, where $\alpha>0$ is small.
Let $f$ be the function defined above,
and let $g$ be the corresponding function for 
the unit ball. By the triangle inequality,
\begin{align}
\notag
f(x)-f(0)  &= \int_{A} (|y-x|-|y|)\, dy\\ 
\label{eq:median-proof}
&\ge
\int_{B^n} (|y-x|-|y|)\, dy - \int_{A\Delta B^n}
\bigl||y-x|-|y|\bigr|\, dy\\
\notag
&\ge g(x)-g(0)-\alpha|x|\,.
\end{align}
In dimension $n\ge 2$,
the function $g$ is twice continuously 
differentiable, strictly radially increasing,
and strictly convex.  We find its Hessian
by differentiating under the integral,
\begin{equation}
\label{eq:Hessian}
D^2 g(x) = \int_{B^n}
\frac{1}{|y-x|}P_{(y-x)^\perp}\, dy
=\int_{x+B^n} \frac{1}{|y|}P_{(y)^\perp}\, dy\,.
\end{equation}
Here $P_{y^\perp}$ denotes the matrix
of the orthogonal projection
onto to hyperplane normal to $y$ (for $y\ne 0$).
The integral converges
and defines a positive definite matrix that depends 
continuously on~$x$.  In particular,
$$
\sum_{j=1}^n \partial_i^2 g(0)= (n-1)
\int_{B^n} \frac{1}{|y|} \, dy = |B^n|\,.
$$
By radial symmetry, $D^2 g(0) = \frac1n |B^n|I$.
For $|x|\le \frac1n$, we use that
$x+B^n\supset \frac{n-1}{n}B^n$
Eq.~\eqref{eq:Hessian} to conclude 
$$
D^2 g(x) \ \ge\  \int_{\frac{n-1}{n}B^n}
\frac{1}{|y|}P_{(y)^\perp}\, dy\ \ge \ 
\left(\frac{n-1}{n}\right)^{n-1} \!\!D^2 g(0)
\ \ge \ \frac{|B^n|}{ne} \,I\,,\qquad (|x|\le \tfrac1n)
$$
as quadratic forms, and thus $g(x) -g(0) \ge \frac{|B^n|}{2 n e } \, |x|^2$.
By Eq.~\eqref{eq:median-proof}, this implies
$$
f(x)-f(0)\ \ge \Bigl( \tfrac{|B^n|}{2 n e}\, |x|-\alpha \Bigr)\,|x|\,,
\qquad(|x|\le \tfrac1n)\,.
$$
If $\alpha\le \frac{|B^n|}{2 n^2 e}$, then $f(x)\ge f(0)$ for
$|x|=\frac1n$, and by convexity
for all $|x|\ge \frac1n$. 
In that case, the minimal value of $f$ lies below
$f(0)$, and hence $|x_0|< \frac{2ne}{|B^n|}\alpha$. 
\end{proof}

\begin{proof}[Proof of Proposition~\ref{prop:reduce}]
Given a set $A$ of asymmetry $\alpha(A)=\alpha>0$
and deficit $\delta(A)=\delta$.
We may assume that $A$ is scaled and centered. 
and that $\alpha$ is small.  
The set $\tilde A$ will be constructed in three steps.
First, the portion of $A$ that lies
in the complement of a ball $(1+R)B^n$ is moved
into a narrow annulus $(1+r)B^n\!\setminus \!(1+\rho)B^n$,
to create a set $A'$.  Then the portion of $A'$
in the annulus $(1-\rho)B^n\!\setminus \! (1-r)B^n$
is moved into the ball $(1-R)B^n$ to create $A''$.
Here, $R=C\alpha^{\frac{\lambda}{n}}$, where $C$ will
be determined below. Once $C$ has been chosen, we take
$\alpha$ small enough that $R\le\frac12$, and
set $\rho=2\alpha/|B^n|$. 
By construction, $(1-R)B^n\subset A\subset (1+R)B^n$,
as required by (P3).
Finally, we perform a translation that re-centers $A''$ 
to $\tilde A$ to obtain Property (P4).

{\em Step 1.}\ 
Define 
$$
A'= \bigl(A\cap (1+R)B^n\bigr) \cup 
\bigl((1+r)B^n\setminus (1+\rho)B^n\bigr)\,,
$$
where $r\ge \rho $ is uniquely determined
by the condition that $|A'|=|A|$. 
By construction, 
$$
A'\cap (1+\rho B^n)= A\cap (1+\rho B^n)\,, \quad
A'\subset (1+R)B^n\,.
$$
In particular, $|A'\Delta A|\le \tfrac12 \alpha$.
Since $\rho\lesssim \alpha$  and
$$
|B^n| = |A'| \ge |B^n| -\tfrac12\alpha
+ \bigl((1+r)^n-(1+\rho)^n\bigr)|B^n|\,,
$$
it follows that $r\lesssim \alpha$.
Consider the  expansion for $\delta(A)-\delta(A')$
from Eq.~\eqref{eq:VVWW}.
By Lemmas~\ref{lem:reduce-V} and~\ref{lem:reduce-W},
\begin{equation}
\label{eq:reduce-VW}
\begin{split}
&\V(A)-\V(A')\gtrsim (R-r)|A'\Delta A|\gtrsim C\alpha^{\frac{\lambda}{n}}
|A'\Delta A|\,,\\
&|\W(A)-\W(A')|\lesssim \alpha^{\frac{\lambda}{n}}|A'\Delta A|\,.
\end{split}
\end{equation}
By choosing $C$ is sufficiently large, we ensure
that $\V(A)-\V(A')\ge |\W(A)-\W(A')|$, and consequently 
$\delta(A')\le \delta$.
Since the implied constants in Eq.~\eqref{eq:reduce-VW}
depend only on $n$ and $\lambda$,
the same is true for $C$.  Moreover, by Lemma~\ref{lem:rigid}, 
$A'$ is centered and $\alpha(A')=\alpha$.

{\em Step 2.}\ 
Define
$$
A'':= (A'\cup (1-R)B^n) 
\setminus\bigl((1-\rho)B^n\setminus (1-r)B^n\bigr)\,,
$$
where $r\ge \rho$ is uniquely determined by
requiring that $|A''|=|A'|$.  
By construction, 
$$
A''\setminus (1-\rho) B^n= A'\setminus (1-\rho) B^n\,, \quad
(1-R)B^n \subset A''\subset (1+R)B^n\,.
$$
In particular, $|A'' \Delta A'|\le \tfrac12 \alpha$.
Since $\rho \lesssim \alpha$  and
$$
|B^n| = |A''| \le \left(1-(1-\rho)^n + (1-r)^n \right)|B^n|+ 
\tfrac12 \alpha\,,
$$
it follows that $r\lesssim \alpha$.
As in Step 1, Lemma~\ref{lem:reduce-V} and~\ref{lem:reduce-W}
imply that $\delta(A'')\le \delta(A')\le\delta$
for $C$ sufficiently large.
Moreover, by Lemma~\ref{lem:rigid}, 
$A''$ is centered and $\alpha(\tilde A)=\alpha$.
Setting 
\begin{equation}
\label{eq:eps}
\eps=-\log(1\!-\!R)\lesssim\alpha^{\frac{\lambda}{n}}\,,
\end{equation}
we see that $A''$ satisfies (P1)-(P3).

{\em Step 3.}\ By Lemma~\ref{lem:median}, there exists
$x_0\in\RR^n$ such that
$$
\int_{x_0+A''} \frac{x}{|x|}\, dx = 0\,.
$$
Setting $\tilde A=x_0+A''$ yields Property (P3).
Since $|x_0|\lesssim \alpha$, 
Property (P3) remains in force
after replacing $R$ with $R+|x_0|$
and adjusting $\eps$ according to Eq.~\eqref{eq:eps}.
\end{proof}


\section{Proof of Theorem~\ref{thm:main}}
\label{sec:main}

Given a subset $A\subset \RR^n$ of asymmetry $\alpha(A)=\alpha$
and deficit $\delta(A)=\delta$, we need
to show that $\delta\gtrsim\alpha^2$. 
We may assume that $\alpha$ is small,
and that $A$ is scaled to have volume $|A|=|B^n|$.
By Proposition~\ref{prop:reduce}, we may further assume
that $A$ is squeezed between two balls as in 
Eq.~\eqref{eq:ring}, and that $\int_A \frac{x}{|x|}\, dx=0$.

Define the functions $M_+$ and $M_-$ by Eq.~\eqref{eq:def-M}.
Since $|A|=|B^n|$,
$$
\int_{S^{n-1}} M_+(\xi)\, d\xi = |A\setminus B^n| 
= \frac12 |A\Delta B^n| \ge \frac{\alpha}{2}\,,
$$
and correspondingly for $M_-$.
This verifies the first line of the constraints in
Eq.~\eqref{eq:constraints}.
For the second line, we compute in polar coordinates
\begin{align*}
\int_{S^{n-1}} \xi \bigl(M_+(\xi)-M_-(\xi)\bigr)\, d\xi 
&= \int_{S^n-1}\int_0^\infty 
\xi\bigl(\chi_A(r\xi)- \chi_{B^n}(r\xi)
\bigr)\, r^{n-1}\, drd\xi\\
&=\int_A \frac{x}{|x|}\, dx=0\,.
\end{align*}

As described in Section~\ref{sec:outline}, we split
the deficit into the first and second variation, 
$\delta=\V-\W$, and then compare these with
the corresponding spherical integrals $V$ and $W$.
By Proposition~\ref{prop:V},
$$\V \geq V(M_+,M_-) + O(\eps) \bigl(\|M_+\|^2 + \|M_-\|^2\bigr)\,,$$
and by Proposition~\ref{prop:W},
$$ \W \leq W(M_+\!-\!M_-) + o(1) \bigl(\|M_+\|^2 + \|M_-\|^2\bigr)
$$
as $\eps\to 0$. Since $\eps\lesssim \alpha^{\frac{\lambda}{n}}$,
the $o(1)$ error term converges to zero uniformly
as $\alpha\to 0$.
By the analysis of the toy model in Eq.~\eqref{eq:toy},
$$
V(M_+,M_-)-W(M_+\!-\!M_-)\ge (\beta_1-\beta_2)\bigl(\|M_+\|^2+\|M_-\|^2\bigr)\,,
$$
where $\beta_1>\beta_2>0$ are determined by $n$ and $\lambda$
through Eq.~\eqref{eq:FH}..
It follows that
$$
\delta 
\ge \bigl(\beta_1-\beta_2-o(1)\bigr)\bigl(\|M_+\|^2 + \|M_-\|^2\bigr)
$$
as $\alpha\to 0$.  Finally, by Schwarz' inequality and 
Eq.~\eqref{eq:constraints},
$ \|M_+\|^2+\|M_-\|^2 
\ge \frac{\alpha^2}{2|S^{n-1}|}$.
\hfill$\Box$

\newpage

\bibliographystyle{amsplain}
\bibliography{Coulomb}

\end{document}